\documentclass[10pt]{article}

\usepackage{bm}
\usepackage{geometry}
\usepackage{amsmath}
\usepackage{amsfonts}
\usepackage{amssymb}
\usepackage{amsthm}
\usepackage{mathrsfs}
\usepackage{authblk}
\usepackage{hyperref}
\usepackage{microtype}
\usepackage{enumerate}

\numberwithin{equation}{section}

\hypersetup{hidelinks}

\newtheorem{theorem}{Theorem}

\newtheorem{lemma}{Lemma}

\newtheorem{assumption}{Assumption}
\newtheorem*{maintheorem}{Main Theorem}

\DeclareMathOperator{\dive}{div}

\DeclareMathOperator{\loc}{loc}

\DeclareMathOperator{\dist}{dist}

\DeclareMathOperator{\trace}{Tr}

\begin{document}
    \title{On Kato's conditions for the inviscid limit of the two-dimensional stochastic Navier-Stokes equation.}
    \author{Ya-Guang Wang}
    \affil{School of Mathematical Sciences, Center for Applied Mathematics, MOE-LSC and SHL-MAC, Shanghai Jiao Tong University, 200240 Shanghai, China; ygwang@sjtu.edu.cn}
    \author{Meng Zhao\footnote{corresponding author}}
    \affil{School of Mathematical Sciences, Shanghai Jiao Tong University, 200240 Shanghai, China; mathematics\_zm@sjtu.edu.cn}

    \date{}
	\maketitle
	\fontsize{12}{15}
	\selectfont
    \begin{abstract}
        We study the asymtotic behavior of solutions to the two-dimensional stochasitc Navier-Stokes (SNS) equation  in the small viscosity limit. The SNS equation is supplemented with no-slip boundary condition, in which a strong boundary layer shall appear in the limit due to the mismatch of the boundary conditions of the SNS equation and the corresponding limit problem. Several equivalent dissipation conditions are derived to ensure the convergence hold in the energy space. One novelty of this work is that we do not assume any smallness for the noise.
    \end{abstract}
	\tableofcontents
    \newpage
    \section{Introduction}
    In this paper, we address the asymtotic behavior of solutions to the two-dimensional stochastic Navier-Stokes equation in small viscosity limit. Consider the following initial-boundary value problem in $\{(t,x)| t>0, x\in D\}$
    \begin{align}\label{1}
        \begin{cases}
    d u^{\nu}(t)+((u^{\nu}\cdot\nabla) u^{\nu}+\nabla p^{\nu}-\nu\Delta u^{\nu})dt=dW_Q(t)\\
    \dive u^{\nu}=0\\
    u^{\nu}(0)=u_0\\
    u^{\nu}|_{\partial D}=0
    \end{cases},
    \end{align}
    where $D\subset \mathbb{R}^2$ is a bounded domain with smooth boundary $\partial D$, $\nu$ is the kinetic viscosity, $u^{\nu}$ is the velocity field, $p^{\nu}$ is the pressure and $W_Q$ denotes the Wiener process taking values in some function spaces specified later. If we formally set the kinetic viscosity $\nu$ to $0$, then formally we arrive at the following problem for the two-dimensional stochastic Euler equation:
    \begin{align}\label{2}
    \begin{cases}
    d u^E(t)+((u^{E}\cdot\nabla) u^{E}+\nabla p^{E})dt=dW_Q(t)\\
    \dive u^E=0\\
    u^E(0)=u_0\\
    (u^E\cdot n)|_{\partial D}=0
    \end{cases}.
    \end{align}
    Here, $n$ is the outward normal vector to the boundary $\partial D$. The proposal of this paper is to derive a necessary and sufficient condition ensuring the limit
    \[\lim_{\nu\to0^+}\|u^{\nu}-u^E\|_{C([0,T];L^2(D))}=0\]
    in probability. 
    
    Historically, the study of the two-dimensional stochastic Navier-Stokes equation was motivated by the theory of turbulence. In order to simulate the complicated dynamics of turbulence, one could input ``genericity" into the Navier-Stokes equation by considering a random force. The original idea could be traced back to Novikov \cite{Novikov}, see also \cite{1988Mathematical}. The wellposedness theory for the stochastic Navier-Stokes equation was first established by Bensousan and Temam in \cite{BT,BENSOUSSAN1973195}, and has been extensively studied since then, e.g. \cite{Flandoli1995MartingaleAS,Bensousan,LIANG2021473}. Compared with the stochastic Navier-Stokes equation, less is known for the stochastic Euler equation, see \cite{GlattHoltz2011LocalAG} and references therein. It is clear that even in the deterministic case, the inviscid limit up to the boundary is impossible to hold in general, due to the mismatch of boundary conditions implemented with the Navier-Stokes and Euler equations respectively. In the seminal paper \cite{Kato1984RemarksOZ}, Kato considered the inviscid limit for the deterministic problem and proved that the validity of 
    \[\lim_{\nu\to0^+}\|u^{\nu}-u^E\|_{C([0,T];L^2(D))}=0\]
    is equivalent to the vanishing of energy dissipation near the boundary
    \begin{align}\label{11}
        \lim_{\nu\to0^+}\int_0^T\|\nabla u^{\nu}\|^2_{L^2(\Gamma_{c\nu})}dt=0,
    \end{align}
    where $\Gamma_{c\nu}$ is a thin neighbourhood with width being $O(\nu)$ of the boundary $\Gamma:=\partial D$. Since then, several modifications of the above energy dissipation condition $(\ref{11})$ have been studied. For example, Wang \cite{Xiaoming} proved that the vanishing of energy involved with only one component of $\nabla u^{\nu}$ near the boundary $\partial D$ is sufficient to ensure the inviscid limit, and Kelliher \cite{Kelliher2007OnKC} replaced the vanishing of $\nabla u^{\nu}$ by the vanishing of vorticity near the boundary $\partial D$. In the case of stochastic inviscid limit, as far as we know, there is only a few work. Luongo \cite{Luongo2021InviscidLF} extended Kato's condition $(\ref{11})$ to
    \[\lim_{\nu\to 0^+}\mathbb{E}\int_0^T\|\nabla u^{\nu}\|^2_{L^2(\Gamma_{c\nu})}dt=0\]
    which ensures the limit
    \[\lim_{\nu\to0^+}\mathbb{E}\|u^{\nu}-u^E\|_{C([0,T];L^2(D))}=0,\]
    where $u^{\nu}$ solves the stochastic Navier-Stokes equation $(\ref{1})$ driven by a ``small" and finite-dimensional noise $\sqrt{v}\sum_{j=1}^N\sigma_j\beta_j(t)$ so that the corresponding inviscid system is the deterministic Euler equation. In \cite{CIPRIANO20152405}, Cipriano and Torrecilla considered the two-dimensional stochastic Navier-Stokes equation driven by an infinite-dimensional Wiener process $W_Q(t)$ implemented with the Navier-slip boundary condition, and proved that 
    \[\lim_{\nu\to0^+}\|u^{\nu}-u^E\|_{C([0,T];L^2(D))}=0\]
    $\mathbb{P}$-almost surely without any energy dissipation conditions. This is because the boundary effect in the case of Navier-slip boundary condition is ``weak". Indeed, by applying the method of multi-scale analysis, one would discover that the boundary layer corrector appears in the next order in the formal expansion of $u^{\nu}$, which makes the validation of the inviscid limit possible, e.g. see \cite{Wang2010BoundaryLI,Iftimie2011ViscousBL} for the corresponding deterministic problem. 
    
    To state our main result, let us introduce the following notations. In what follows, we shall always fix a stochastic basis $(\Omega,\mathscr{F},\mathbb{P},\mathbb{F},\{W_Q(t)\}_{t\ge0})$, where the filtration $\mathbb{F}:=\{\mathscr{F}_t\}_{t\ge 0}$ satisfies the usual condition and the Wiener process $\{W_Q(t)\}_{t\ge 0}$ is adapted to $\mathbb{F}$. Write
    \[H:=\{u\in L^2(D)|\dive u=0,\ u\cdot n|_{\partial D}=0\}.\]
    Let $P: L^2(D)\to H$ denote the Leray projector. Define 
    \[\mathscr{D}(P\Delta):=\left\{u\in H^2(D)\bigcap H_0^1(D)\big| \dive u=0\right\}\]
    endowed with the norm
    \[\|u\|_{\mathscr{D}(P\Delta)}:=\|P\Delta u\|_{L^2(D)},\]
    and inductively for each $k\ge 2$,
    \[\mathscr{D}(P\Delta)^k:=\{u\in \mathscr{D}(P\Delta)^{k-1}| P\Delta u\in \mathscr{D}(P\Delta)^{k-1}\}\]
    endowed with the norm
    \[\|u\|_{\mathscr{D}(P\Delta)^k}:=\|(P\Delta)^k u\|_{L^2(D)}.\]
    As shown in \cite{Guermond2011ANO}, for each $k \ge 1$, we have \[C_k^{-1}\|u\|_{\mathscr{D}(P\Delta)^k}\le \|u\|_{H^{2k}(D)}\le C_k\|u\|_{\mathscr{D}(P\Delta)^k}\]
    for all $u\in \mathscr{D}(P\Delta)^k$. Hence, one could simply view the space $\mathscr{D}(P\Delta)^k$ as a closed subspace of the Sobolev space $H^{2k}(D)\bigcap H$. We need the following assumption:
    \begin{assumption}\label{assumption}\rm
    The initial data $u_0\in H^4(D)\bigcap H$ is non-random and the Wiener process $\{W_Q(t)\}_{t\ge0}$ takes value in the space $\mathscr{D}(P\Delta )^3$ with the covariance operator given by $Q$.
    \end{assumption}
    The main result is presented below.
    \begin{maintheorem}\label{theorem3}
    Let Assumption \ref{assumption} be fulfilled. Assume $u^{\nu}$ and $u^E$ are solutions to the problems $(\ref{1})$ and $(\ref{2})$ respectively, then the following statements are equivalent:
    \begin{enumerate}[(1)]
        \item $\lim_{\nu\to0^+}\|u^{\nu}-u^E\|_{C([0,T];L^2(D))}=0$ in probability;
        \item $\lim_{\nu\to0^+}\nu\int_0^T\|\nabla u^{\nu}\|^2_{L^2(D)}dt=0$ in probability;
        \item there exists a deterministic parameter $\delta_0(\nu)$ such that 
        \[\lim_{\nu\to0^+}\frac{\nu}{\delta_0(\nu)}=0\qquad \lim_{\nu\to0^+}\delta_0(\nu)=0\]
        and one of the following
        \begin{enumerate}
            \item $\lim_{\nu\to0^+}\nu\int_0^T\|\partial_{n} u^{\nu}_{\tau}\|^2_{L^2(\Gamma_{\delta_0})}dt=0$
            \item $\lim_{\nu\to0^+}\nu\int_0^T\|\partial_{n} u^{\nu}_{n}\|^2_{L^2(\Gamma_{\delta_0})}dt=0$
            \item $\lim_{\nu\to0^+}\nu\int_0^T\|\partial_{\tau} u^{\nu}_{\tau}\|^2_{L^2(\Gamma_{\delta_0})}dt=0$
        \end{enumerate}
        holds in probability, where $\Gamma_{\delta_0}$ is a thin neighbourhood of the boundary $\partial D$ with width being $\delta_0$ and the notation such as $\partial_n u_{\tau}^{\nu}$ denotes the normal derivative of the tangetial velocity $\frac{\partial}{\partial n}(u^{\nu}\cdot \tau)$, where $n$ and $\tau$ are proper extensions of the unit normal and tangential vectors respectively of $\partial D$ in a small neighbourhood of $\partial D$.
    \end{enumerate}
    \end{maintheorem}
    Let us explain the main difficulty for studying the stochastic inviscid limit problem, especially compared with its deterministic counterpart. To remedy the disparity of boundary conditions for problems $(\ref{1})$ and $(\ref{2})$, Kato constructed an artificial corrector $v$ in \cite{Kato1984RemarksOZ} such that $v|_{\partial D}=u^E|_{\partial D}$ and $v$ is supported in a thin domain near the boundary, which ensures the smallness of $v$. The corrector $v$ could viewed as $v=Tu^E$, where $T$ is a linear differential operator of order one. Hence, if one repeats the argument of Kato and thereby sets
    \[w:=u^{\nu}-u^E+Tu^E,\]
    then the dynamics of $w$ is given by
    \[dw=(P((u^{E}\cdot\nabla)u^{E})-P((u^{\nu}\cdot\nabla)u^{\nu})+\nu P\Delta u^{\nu}-TP((u^{E}\cdot\nabla)u^{E}))dt+TdW_Q(t).\]
    Notice that the lack of differentiablity of $w$ with respect to the time variable would prevent us from applying the classical energy method. Instead, to establish estimate for $w$, one has to use the It\^o formula to obtain
    \begin{align*}
        &\|w\|^2_{L^2}=2\int_0^t\langle w,TdW_Q(s)\rangle_{L^2}\\&+2\int_0^t\langle w,P((u^{E}\cdot\nabla)u^{E})-P((u^{\nu}\cdot\nabla)u^{\nu})+\nu P\Delta u^{\nu}-TP((u^{E}\cdot\nabla)u^{E})\rangle_{L^2}ds\\&+t\trace TQT^*.
    \end{align*}
    Notice that $t\trace TQT^{*}$ comes from the diffusion part of the boundary corrector $v$ and is independent of the viscosity, though $T$ possesses certain type of smallness, that is, $Tu$ is supported in a thin domain for all smooth divergence free vector field $u$. The lack of control on $t\trace TQT^{*}$ illustrates the main difficulty for studying the stochastic inviscid limit problem. To overcome this difficulty, we split $\{u^{\nu}(t)\}_{t\ge0}$ into two parts 
    \begin{align}\label{13}
    u^{\nu}(t)=v^{\nu}(t)+z^{\nu}(t),
    \end{align}
    where $z^{\nu}$ solves the following problem for the stochastic linear Stokes equation
    \begin{align}\label{3}
    \begin{cases}
    d z^{\nu}(t)+(\nabla p^{\nu}_1-\nu\Delta z^{\nu})dt=dW_Q(t)\\
    \dive z^{\nu}=0\\
    z^{\nu}(0)=0\\
    z^{\nu}|_{\partial D}=0
    \end{cases}
    \end{align}
    and $v^{\nu}$ solves the problem of the NS-type PDE with random parameter:
    \begin{align}\label{4}
    \begin{cases}
    \partial_t v^{\nu}(t)+(u^{\nu}\cdot\nabla) u^{\nu}+\nabla (p^{\nu}-p^{\nu}_1)-\nu\Delta v^{\nu}=0\\
    \dive v^{\nu}=0\\
    v^{\nu}(0)=u_0\\
    v^{\nu}|_{\partial D}=0.
    \end{cases}
    \end{align}
    The split method enables us to deal with the stochastic nature and the boundary effect separately. Indeed, we analyse the system $(\ref{3})$ with tools from the theory of semigroups and apply the classical Kato-type argument on the system $(\ref{4})$. 
    
    \section{Preliminaries}
    \subsection{Wellposedness of problems (\ref{1}) and (\ref{2})}
    The following wellposedness result of $(\ref{1})$ is well-known. As for a proof, we refer to \cite{kuksin2012mathematics}. 
    \begin{theorem}
    Let Assumption \ref{assumption} be fulfilled. Then, there exists a unique solution $\{u^{\nu}(t)\}_{t\ge0}$ to $(\ref{1})$ such that 
    \begin{enumerate}[(1)]
        \item almost every trajectory of $\{u^{\nu}(t)\}_{t\ge0}$ belongs to the space 
        \[\chi:=C(\mathbb{R}_+;H)\bigcap L_{\loc}^2\left(\mathbb{R}_+,V\right),\]
        where $\mathbb{R}_+=[0,\infty)$ and $V:=H_0^1(D)\bigcap H$;
        \item the $H$-valued process $\{u^{\nu}(t)\}_{t\ge0}$ is adapted to the filtration $\mathbb{F}$;
        \item the system $(\ref{1})$ holds in the sense that 
        \[\mathbb{P}\left[u^{\nu}(t)+\int_0^t P((u^{\nu}\cdot \nabla )u^{\nu})-\nu P\Delta u^{\nu}ds=u_0+W_Q(t),\ \forall t\ge0\right]=1.\]
    \end{enumerate}
    The above uniqueness of the solution $u^{\nu}$ holds in the sense of indistinguishability, that is, if $\{\tilde{u}^{\nu}(t)\}_{t\ge0}$ is another solution of $(\ref{1})$, then
    \begin{align*}
        \mathbb{P}[\tilde{u}^{\nu}(t)=u^{\nu}(t),\ \forall t\ge0]=1.
    \end{align*}
    \end{theorem}
    The following wellposedness theory for $(\ref{2})$ is given in \cite{GlattHoltz2011LocalAG}.
    \begin{theorem}\label{theorem2}
    Let Assumption \ref{assumption} be fulfilled. Then, there exists a unique solution $\{u^{E}(t)\}_{t\ge0}$ to $(\ref{2})$ such that
    \begin{enumerate}[(1)]
        \item almost every trajectory of $\{u^{E}(t)\}_{t\ge0}$ belongs to the space $C(\mathbb{R}_+;H^4\bigcap H)$;
        \item $\{u^{E}(t)\}_{t\ge0}$ is adapted to the filtration $\mathbb{F}$;
        \item the system $(\ref{2})$ holds in the sense that 
        \[\mathbb{P}\left[u^{E}(t)+\int_0^t P((u^{E}\cdot \nabla) u^{E})ds=u_0+W_Q(t),\ \forall t\ge0\right]=1,\]
    \end{enumerate}
    and the uniqueness of the solution $u^E$ holds in the sense of indistinguishability.
    \end{theorem}
    \subsection{Basic estimates}
    Let us recall some basic nonlinear estimates needed in our following analysis. For $m>\frac{2}{p}$, by applying the Galiardo-Nirenberg inequality, we have the following Moser-type estimate
    \[\|uv\|_{W^{m,p}(D)}\lesssim\|u\|_{L^{\infty}(D)}\|v\|_{W^{m,p}(D)}+\|v\|_{L^{\infty}(D)}\|u\|_{W^{m,p}(D)}\qquad \forall u,v\in W^{m,p}(D).\]
    We recall some estimates for the Leray projector $P$. For $v\in L^2(D)$, we have $Pv=(1-Q)v$, where 
    \[Qv=-\nabla \pi\]
    for any $\pi\in H^1(D)$ which solves the Neumann problem 
    \[\begin{cases}
    -\Delta \pi=\dive v\qquad &\mbox{in $D$}\\
    \frac{\partial\pi}{\partial n}=-v\cdot n\qquad &\mbox{on $\partial D$}.
    \end{cases}\]
    Moreover, if $v\in H^m(D)$ for a fixed $m\ge1$, then according to the regularity estimates of the Neumann elliptic problem, we have 
    \[\|\nabla \pi\|_{H^{m}(D)}\lesssim\|\dive v\|_{H^{m-1}(D)}+\|v\cdot n\|_{H^{m-\frac{1}{2}}(\partial D)}\lesssim\|v\|_{H^{m}(D)},\]
    which implies $\|Pv\|_{H^{m}(D)}\lesssim \|v\|_{H^{m}(D)}$, that is, the Leray projector $P$ is a bounded linear operator on $H^{m}(D)$. Hence, for $u\in H^{4}(D)$, it follows that
    \[\|P((u\cdot\nabla)u)\|_{H^{3}(D)}\lesssim\|u\|_{L^{\infty}(D)}\|\nabla u\|_{H^3(D)}+\|\nabla u\|_{L^{\infty}(D)}\|u\|_{H^{3}(D)}\lesssim \|u\|^2_{H^4(D)}.\]
    Combining the above estimate with Theorem $\ref{theorem2}$, we conclude that $\mathbb{P}$-almost every trajectory of $P(u^E\cdot\nabla)u^E$ belongs to $C(\mathbb{R}_+;H^3(D)\bigcap H)$. 
    Let 
    \begin{align}\label{14}
    v^E(t):=u^{E}(t)-W_Q(t).\end{align}
    Then, $\mathbb{P}$-almost surely we have 
    \[v^{E}(t)-v(0)=-\int_0^t P((u^{E}\cdot \nabla) u^{E})ds\qquad \forall t\ge0.\]
    This implies that $\mathbb{P}$-almost every trajectory of $v^E$ is differentiable and we have 
    \[\mathbb{P}[\partial_t v^E(t)=-P(u^E(t)\cdot\nabla)u^E(t),\ \forall t\ge0]=1.\]
    In particular, it follows that
    \[\sup_{t\in[0,T]}\|\partial_t v^E (t)\|_{C^1(\overline{D})}<\infty\]
    $\mathbb{P}$-almost surely. Finally, the spatial regularity of $u^E$ and $W_Q$ ensures that 
    \begin{align*}\sup_{t\in[0,T]}\| v^E(t)\|_{C^2(\overline{D})}<\infty\end{align*}
    $\mathbb{P}$-almost surely. In a word, $\mathbb{P}$-almost every trajectory of $v^E$ belongs to the space \begin{align}\label{aaaaa}\chi_1:=C^1\left(\mathbb{R}_+;H^3(D)\bigcap H\right)\bigcap C\left(\mathbb{R}_+;H^4(D)\bigcap H\right)\end{align}
    and solves the following Euler-type PDE with random parameter:
    \[\begin{cases}
    \partial_t v^{E}(t)+P(u^E(t)\cdot \nabla)u^E(t)=0\\
    v^E(t)=u^E(t)-W_Q(t)\\
    v^E(0)=u_0
    \end{cases}.\]
    \subsection{The boundary corrector}\label{aaa}
    In this sub-section, we construct an artificial boundary corrector to overcome the mismatch of boundary conditions between $u^{\nu}$ and $u^E$. We mainly follow the strategy given by Wang \cite{Xiaoming}. Suppose the boundary $\partial D$ is parameterized by 
    $s\mapsto (\gamma_1(s),\gamma_2(s))$, $s\in [0,1)$. Let $\delta$ be a small quantity which will be specified later. In the thin neighbourhood $\Gamma_{\delta}$ of the boundary $\partial D$, we introduce the following local coordinate. For $x\in\Gamma_{\delta}$, we set $P(x)$ as the closest point to $x$ on the boundary $\partial D$, and define $s(x)$ as the parameter of $\partial D$ corresponding to $P(x)$. Write
    \[\alpha(x):=\dist (x,P(x)).\]
    Clearly, $(s(x),\alpha(x))$ forms a local coordinate in $\Gamma_{\delta}$, where $s(x)$ is the coordinate along the tangential direction
    \[\tau:=(\gamma_1^{'}(s),\gamma_2^{'}(s)),\]
    and $\alpha(x)$ is the coordinate along the normal direction
    \[n:=(-\gamma_2^{'}(s),\gamma_1^{'}(s)).\]
    Let $\xi:[0,\infty)\to [-1,1]$ be a smooth cutoff function with $\xi(r)\equiv1$ near $r=0$ and $\xi(r)\equiv0$ for $r\ge1$ and
    \[\int_0^{\infty}\xi(r)dr=0\]
    Set
    \begin{align}\label{7}
        v(t,x):=-\nabla^{\perp}_x\left(v^E_{\tau}|_{\partial D}(s(x),t)\int_0^{\alpha(x)}\xi\left(\frac{r}{\delta}\right)dr\right),
    \end{align}
    where $\nabla_x^{\perp}:=(-\partial_{x_2},\partial_{x_1})$. Notice that 
    \begin{align}\label{0aa}\nabla_x=h^{-1}(s(x),\alpha(x))\tau(x) \partial_{\tau}+n(x)\partial_{n}\end{align}
    and 
    \begin{align}\label{0a}\partial_{\tau}=h(s(x),\alpha(x))\tau(x)\cdot\nabla_x,\qquad \partial_n=n(x)\cdot\nabla_x, \end{align}
    where 
    \begin{align}\label{0aaa}h(s(x),\alpha(x)):=1-\alpha(x)(\gamma_1^{'}(s(x))\gamma_2^{''}(s(x))-\gamma_2^{'}(s(x))\gamma_1^{''}(s(x))).\end{align}
    This implies
    \begin{align}\label{bb}v_{\tau}=\tau\cdot v=v_{\tau}^E|_{\partial D}(s(x),t)\xi\left(\frac{\alpha(x)}{\delta}\right)\end{align}
    and
    \begin{align}\label{bbb}v_{n}=n\cdot v=-\frac{1}{h(s(x),\alpha(x))}\partial_{\tau}v^{E}_{
    \tau}|_{\partial D}(s(x),t)\int_0^{\alpha(x)}\xi\left(\frac{r}{\delta}\right)dr,\end{align}
    and thus 
    \[v|_{\partial D}=v^E|_{\partial D}.\] 
    Furthermore, $v$ is supported in $\Gamma_{\delta}$ and we have 
    \begin{align}
        \partial_{\tau}v_{\tau}&=\partial_{\tau}v_{\tau}^E|_{\partial D}(s(x),t)\xi\left(\frac{\alpha(x)}{\delta}\right)\label{a}\\
        \partial_{n}v_{\tau}&=\frac{1}{\delta}v_{\tau}^E|_{\partial D}(s(x),t)\xi^{'}\left(\frac{\alpha(x)}{\delta}\right)\label{aa}\\
        \partial_{\tau}v_{n}&=-\partial_{\tau}\left(h^{-1}(s(x),\alpha(x))\partial_{\tau}v_{\tau}^E|_{\partial D}(s(x),t)\int_0^{\alpha(x)}\xi\left(\frac{r}{\delta}\right)dr\right)\label{aaa*}\\
        \partial_{n}v_{n}&=-\partial_{n}\left(h^{-1}(s(x),\alpha(x))\partial_{\tau}v_{\tau}^E|_{\partial D}(s(x),t)\int_0^{\alpha(x)}\xi\left(\frac{r}{\delta}\right)dr\right)\label{aaaa},
    \end{align}
    Combining $(\ref{aaaaa})$, $(\ref{a})$, $(\ref{aa})$, $(\ref{aaa*})$ and $(\ref{aaaa})$, we immediately obtain the following estimates: 
    \[\|v\|_{C([0,T];L^2(D))}\le C(\omega)\delta^{\frac{1}{2}},\qquad \|\partial_t v\|_{C([0,T];L^2(D))}\le C(\omega)\delta^{\frac{1}{2}},\]
    \[\|\nabla v\|_{C([0,T];L^2(D))}\le C(\omega)\delta^{-\frac{1}{2}},\qquad \| v\|_{C([0,T]\times \bar{D})}\le C(\omega),\]
    \[\| \nabla v\|_{C([0,T]\times \bar{D})}\le C(\omega)\delta^{-1},\qquad \|\rho\nabla v\|_{C([0,T]\times \bar{D})}\le C(\omega), \]
    \[\| \rho^2\nabla v\|_{C([0,T]\times \bar{D})}\le C(\omega)\delta,\]
    where
    \[\rho(x):=\dist(x,\partial D).\]
    We remark that these estimates hold $\mathbb{P}$-almost surely and the constant $C$ depends on $\omega\in\Omega$ via its dependence on $\|v\|_{C^{1}([0,T]\times \overline{D})}$. In what follows, we shall frequently use the notation $C(\omega)$ to denote a positive random variable which is finite $\mathbb{P}$-almost surely and does not depend on $\nu, t$ and $x$. 
    \section{Proof of the Main Theorem}
    \subsection{Proof of the Main Theorem \textit{(2)} from \textit{(1)}}
    Let us start with the implication from the statement \textit{(1)} to the statement \textit{(2)} given in the Main Theorem. By applying It\^o's formula to $\|u^{\nu}(t)\|^2_{L^2(D)}$ and $\|u^{E}(t)\|^2_{L^2(D)}$, we derive that
    \begin{align}\label{5}
        \|u^{\nu}(t)\|^2_{L^2(D)}+2\nu\int_0^t&\|\nabla u^{\nu}(s)\|^2_{L^2(D)}ds\notag\\&=\|u_0\|^2_{L^2(D)}+t\trace Q_0+2\int_0^t\langle u^{\nu}(s),d W_Q(s)\rangle_{L^2(D)}
    \end{align}
    and 
    \begin{align}\label{6}
        \|u^{E}(t)\|^2_{L^2(D)}=\|u_0\|^2_{L^2(D)}+t\trace Q_0+2\int_0^t\langle u^{E}(s),d W_Q(s)\rangle_{L^2(D)}
    \end{align}
    hold in the sense of indistinguishability. Here, we view $\{W_Q(t)\}_{t\ge0}$ as an $H$-valued Wiener process because of the natural embedding
    \[\mathscr{D}((-P\Delta)^3)\hookrightarrow H,\]
    and by $Q_0$, we denote the covariance operator of this $H$-valued Wiener process $\{W_{Q_0}(t)\}_{t\ge0}$. We claim that
    \begin{align}\label{12}
    \lim_{\nu\to0^+}\sup_{t\in[0,T]}\left|\int_0^t\langle u^{\nu}(s)-u^E(s),d W_Q(s)\rangle_{L^2(D)}\right|^2=0\end{align}
    in probability. Indeed, let us set $\tau_{\nu}:=\inf\{t\ge0|\int_0^t\|u^{\nu}(s)-u^E(s)\|^2_{L^2(D)}ds\ge 1\}.$ Since $\|u^{\nu}-u^{E}\|_{C([0,T];L^2(D))}\to 0$ in probability, then we have 
    \[\lim_{\nu\to 0^+}\int_0^T\|u^{\nu}(s)-u^E(s)\|^2_{L^2(D)}ds=0\]
    in probability, which implies
    \begin{align*}
        \mathbb{P}[\tau_{\nu}< T]\le\mathbb{P}\left[\int_0^T\|u^{\nu}(s)-u^E(s)\|^2_{L^2(D)}ds\ge 1\right]\longrightarrow 0,
    \end{align*}
    as $\nu\to0^+$.
    By applying the Burkholder-Gundy-Davis inequality and dominated convergence theorem, we derive 
    \begin{align*}
        \mathbb{E}\sup_{t\in[0,T\wedge\tau_{\nu}]}&\left|\int_0^t\langle u^{\nu}(s)-u^E(s),d W_Q(s)\rangle_{L^2(D)}\right|^2\\&\lesssim \mathbb{E}\int_0^{T\wedge\tau_{\nu}}\|u^{\nu}(s)-u^E(s)\|_{L^2(D)}^2ds\longrightarrow 0
    \end{align*}
    as $\nu\to 0^+$, which yields 
    \[\lim_{\nu\to0^+}\sup_{t\in[0,T\wedge\tau_{\nu}]}\left|\int_0^t\langle u^{\nu}(s)-u^E(s),d W_Q(s)\rangle_{L^2(D)}\right|^2=0\]
    in probability. Moreover, for a fixed $0<\epsilon<1$, we have
    \begin{align*}
        &\left[\sup_{t\in[0,T]}\left|\int_0^t\langle u^{\nu}(s)-u^E(s),d W_Q(s)\rangle_{L^2(D)}\right|^2\ge\epsilon\right]\\&=\left(\left[\sup_{t\in[0,T]}\left|\int_0^t\langle u^{\nu}(s)-u^E(s),d W_Q(s)\rangle_{L^2(D)}\right|^2\ge\epsilon\right]\bigcap[\tau_{\nu}\ge T]\right)
        \\&\bigcup\left(\left[\sup_{t\in[0,T]}\left|\int_0^t\langle u^{\nu}(s)-u^E(s),d W_Q(s)\rangle_{L^2(D)}\right|^2\ge\epsilon\right]\bigcap[\tau_{\nu}< T]\right)
        \\&\subset \left[\sup_{t\in[0,T\wedge\tau_{\nu}]}\left|\int_0^t\langle u^{\nu}(s)-u^E(s),d W_Q(s)\rangle_{L^2(D)}\right|^2\ge\epsilon\right]\bigcup[\tau_{\nu}< T].
    \end{align*}
    Therefore, we obtain
    \begin{align*}
        \mathbb{P}&\left[\sup_{t\in[0,T]}\left|\int_0^t\langle u^{\nu}(s)-u^E(s),d W_Q(s)\rangle_{L^2(D)}\right|^2\ge\epsilon\right]\\&\le \mathbb{P}[\tau_{\nu}< T]+\mathbb{P}\left[\sup_{t\in[0,T\wedge\tau_{\nu}]}\left|\int_0^t\langle u^{\nu}(s)-u^E(s),d W_Q(s)\rangle_{L^2(D)}\right|^2\ge\epsilon\right]\longrightarrow 0
    \end{align*}
    as $\nu\to 0^+$. This verifies our claim $(\ref{12})$. Now set $t=T$ and let $\nu\to0^+$ in $(\ref{5})$, we derive 
    \begin{align*}
        \|u^{E}(T)\|^2_{L^2(D)}+2\limsup_{\nu\to0^+}&\nu\int_0^T\|\nabla u^{\nu}(s)\|^2_{L^2(D)}ds\\&=\|u_0\|^2_{L^2(D)}+T\trace Q_0+2\int_0^T\langle u^{E}(s),d W_Q(s)\rangle_{L^2(D)}.
    \end{align*}
    By comparing with $(\ref{6})$, this implies 
    \[\lim_{\nu\to0^+}\nu\int_0^T\|\nabla u^{\nu}(s)\|^2_{L^2(D)}ds=0\]
    in probability as desired.
    \subsection{Proof of the Main Theorem \textit{(1)} from \textit{3(a)}}
    It is clear to have the assertion $\textit{(3)}$ from $\textit{(2)}$ given in the Main Theorem. We turn to the implication from $\textit{3(a)}$ to $\textit{(1)}$, and the proof is divided into the following two parts.
    
    \subsubsection{Analysis for the stochastic linear Stokes equation (\ref{3})} 
    First, let us recall the following facts, cf. \cite{semigroup} and \cite{Brezis}.
    \begin{enumerate}
        \item The Stokes operator $P\Delta$ generates a holomorphic semigroup $\{e^{tP\Delta}\}_{t\ge0}$ on $H$.
        \item The restriction of the semigroup $\{e^{tP\Delta}\}_{t\ge0}$ on $\mathscr{D}(P\Delta)^k$ forms a holomorphic semigroup with the generator given by $(P\Delta)|_{\mathscr{D}(P\Delta)^{k+1}}$.
    \end{enumerate}
    Before studying the problem $(\ref{3})$, let us first give the following abstract inviscid limit result.
    \begin{lemma}\label{lemma1}
    Let $\{S(t)\}_{t\ge0}$ be a holomorphic semigroup generated by $A: \mathscr{D}(A)\to H$, where $H$ is a separable Hilbert space. Suppose that $\{W_Q(t)\}_{t\ge0}$ is a $Q$-Wiener process which takes value in the space $\mathscr{D}(A)$ (equipped with the graph norm). Then, we have
    \[\lim_{\nu\to 0^+}\mathbb{E}\|X^{\nu}-W_Q\|_{C([0,T];H)}=0,\]
    where $X^{\nu}$ solves
    \[\begin{cases}
    dX^{\nu}(t)=\nu AX^{\nu}(t)dt+dW_Q(t)\\
    X^{\nu}(0)=0.
    \end{cases}\]
    \end{lemma}
    \begin{proof}
    The solution $X^{\nu}$ is given by the stochastic convolution formula:
    \[X^{\nu}(t)=\int_0^tS^{\nu}(t-s)dW_Q(s)\]
    cf. \cite{daprato}, where $\{S^{\nu}(t)\}_{t\ge0}$ is the holomorphic semigroup generated by $\nu A$. The stochastic integration by parts formula ensures that 
    \[X^{\nu}(t)=W_Q(t)-\nu A\int_0^tS^{\nu}(t-s)W_Q(s)ds.\]
    Moreover, notice that $\{W_Q(t)\}_{t\ge0}$ takes value in $\mathscr{D}(A)$, then we obtain
    \[X^{\nu}(t)=W_Q(t)-\nu\int_0^tS^{\nu}(t-s)AW_Q(s)ds,\]
    which implies
    \begin{align*}
        \mathbb{E}\|X^{\nu}-W_Q\|_{C([0,T];H)}&\le \nu\mathbb{E}\left\|\int_0^tS^{\nu}(t-s)AW_Q(s)ds\right\|_{C([0,T];H)}\\&\le \nu\mathbb{E}\sup_{t\in[0,T]}\int_0^t\|S^{\nu}(t-s)\|_{L(H)}\|AW_Q(s)\|_{H}ds\\&\le \nu\left(\int_0^{T}\|S^{\nu}(t)\|_{L(H)}dt\right)\mathbb{E}\|W_Q\|_{C([0,T];\mathscr{D}(A))}.
    \end{align*}
    Hence, to obtain the desired convergence, it suffices to bound the term 
    \begin{align*}
        I_{\nu}=\int_0^T\|S^{\nu}(t)\|_{L(H)}dt.
    \end{align*}
    To this end, let us introduce the Yoshida approximation
    \[A_n^{\nu}:=n\nu A(nI-\nu A)^{-1}.\]
    Define
    \[e^{tA_n^{\nu}}:=\sum_{k=0}^{\infty}\frac{t^k(A_n^{\nu})^k}{k!}.\]
    Then, we have 
    \[\lim_{n\to\infty}e^{tA_{n}^{\nu}}=S^{\nu}(t)\]
    in the strong topology of operators. This implies that 
    \[\|S^{\nu}(t)\|_{L(H)}\le \liminf_{n\to\infty}\|e^{tA_n^{\nu}}\|_{L(H)}.\]
    Since $A$ is a generator of a holomorphic semigroup, then Hille-Yoshida theorem ensures that there exist constants $M>0$ and $c\in\mathbb{R}$ such that
    \[\|(\lambda I-A)^{-k}\|_{L(H)}\le \frac{M}{(\lambda- c)^k}\]
    for all $k\in\mathbb{N}$ and $\lambda >c$. Moreover, notice that 
    \[A_n^{\nu}=n^2(nI-\nu A)^{-1}-nI,\]
    then we derive 
    \begin{align*}
        \|e^{tA_n^{\nu}}\|_{L(H)}&\le e^{-nt}\|e^{tn^2(nI-\nu A)^{-1}}\|_{L(H)}\\&\le e^{-nt}\sum_{k=0}^{\infty}\frac{t^kn^{2k}}{k!}\|(nI-\nu A)^{-k}\|_{L(H)}\\&\le e^{-nt}\sum_{k=0}^{\infty}\frac{t^kn^{k}}{k!}\left(\frac{n}{\nu}\right)^k\frac{M}{(n/\nu-c)^k}\le M\exp\left(\nu t\frac{cn/\nu }{n/\nu -c}\right),
    \end{align*}
    which implies
    \[\|S^{\nu}(t)\|_{L(H)}\le Me^{\nu ct}.\]
    Hence, we obtain
    \[I_{\nu}\le \int_0^T Me^{\nu ct}dt\le M\int_0^Te^{|c|t}dt\le MTe^{|c|T}<\infty\]
    for all $\nu<<1$. This implies 
    \begin{align*}
        \mathbb{E}\|X^{\nu}-W_Q\|_{C([0,T];H)}\le MTe^{|c|T}\nu \mathbb{E}\|W_Q\|_{C([0,T];\mathscr{D}(A))}\longrightarrow 0,
    \end{align*}
    as $\nu\to 0^+$. We thus conclude the proof.
    \end{proof}
    By applying the Leray projector on both sides of $(\ref{3})$, we derive that 
    \[\begin{cases}
    dz^{\nu}=\nu P\Delta z^{\nu}dt+dW_Q(t)\\
    z^{\nu}(0)=0
    \end{cases}.\]
    Since 
    \[P\Delta|_{\mathscr{D}(P\Delta)^3}: \mathscr{D}(P\Delta)^3\to \mathscr{D}(P\Delta)^2\]
    generates a holomorphic semigroup $e^{tP\Delta}|_{\mathscr{D}(P\Delta)^2}$ and $W_Q$ takes value in $\mathscr{D}(P\Delta)^3$, then by using Lemma \ref{lemma1}, we conclude 
    \[\lim_{\nu\to 0^+}\mathbb{E}\|z^{\nu}-W_Q\|_{C([0,T];H^4(D))}=0.\]
    Moreover, by analysing the proof of Lemma $\ref{lemma1}$, it is clear to have the following pathwise estimate
    \begin{align}\label{10}
        \|z^{\nu}-W_Q\|_{C([0,T];H^4(D))}\le C(\omega)\nu,
    \end{align}
    where 
    \[C(\omega)=MTe^{|c|T}\|W_Q\|_{C([0,T];H^6(D))}.\]

    \subsubsection{Analysis for the NS-type equation (\ref{4})} 
    Since the system $(\ref{4})$ is a NS-type PDE with random parameter, then we could apply the classical Kato-type argument. Set
    \begin{align}\label{15}
    w(t):=v^{\nu}(t)-v^E(t)+v(t),\end{align}
    where $v$ is given by $(\ref{7})$ and the random parameter $\delta(\omega)$ appeared in $(\ref{7})$ is to be determined. Notice that we have 
    \[w|_{\partial D}=(v^E-v)|_{\partial D}=0.\]
    Then, 
    \begin{align*}
        \frac{d}{dt}\|w(t)\|_{L^2(D)}^2&=2\langle w(t),\partial_tw(t)\rangle_{L^2(D)}\\
        &=2\langle w(t),\nu P\Delta v^{\nu}-P(u^{\nu}\cdot \nabla)u^{\nu}+P(u^E\cdot \nabla)u^E+\partial_tv(t)\rangle_{L^2(D)},
    \end{align*}
    which implies 
    \begin{align*}
        \frac{d}{dt}\|w(t)\|_{L^2(D)}^2+2\nu\|\nabla w\|^2_{L^2(D)}=:I_1+I_2+I_3,
    \end{align*}
    where
    \begin{align}\label{a1}
        I_1&:=2\langle w, \nu P\Delta (v^{E}-v)\rangle_{L^2(D)} \\\label{a2}
        I_2&:=2\langle w, P(u^E\cdot\nabla)u^E-P(u^{\nu}\cdot\nabla)u^{\nu}\rangle_{L^2(D)} \\\label{a3}
        I_3&:=2\langle w,\partial_t v\rangle_{L^2(D)}.
    \end{align}
    For $I_1$ given by $(\ref{a1})$, we derive 
    \begin{align}\label{e1}
        I_1&=-2\nu\langle\nabla w,\nabla v^E\rangle_{L^2(D)}+2\nu\langle\nabla w,\nabla v\rangle_{L^2(D)}\notag\\
        &\le 2\nu\|\nabla w\|_{L^2(D)}(\|\nabla v^E\|_{L^2(D)}+\|\nabla v\|_{L^2(D)})\notag\\
        &\le \nu\|\nabla w\|_{L^2(D)}^2+2\nu\|\nabla v^E\|^2_{L^2(D)}+2\nu\|\nabla v\|^2_{L^2(D)}\notag\\
        &\le \nu\|\nabla w\|_{L^2(D)}^2+\nu C(\omega)+\nu\delta^{-1}C(\omega).
    \end{align}
    As for $I_2$ given by $(\ref{a2})$, we have
    \begin{align}\label{a4}
        I_2&=2\langle w, P((u^E-u^{\nu})\cdot\nabla)u^E\rangle_{L^2(D)}+2\langle w,P(u^{\nu}\cdot\nabla)(u^E-u^{\nu})\rangle_{L^2(D)}\notag\\&=:I_{21}+I_{22}.
    \end{align}
    From $(\ref{13})$, $(\ref{14})$ and $(\ref{15})$, we have
    \begin{align}\label{b}u^E-u^{\nu}=W_Q-z^{\nu}-w+v,\end{align}
    which implies
    \begin{align}\label{e21}
        I_{21}&=2\langle w, ((u^E-u^{\nu})\cdot\nabla)u^E\rangle_{L^2(D)}
        \notag\\&\le 2\|w\|_{L^2(D)}\|u^E-u^{\nu}\|_{L^2(D)}\|\nabla u^E\|_{L^{\infty}(D)}
        \notag\\&\le C(\omega)\|w\|_{L^2(D)}(\|W_{Q}-z^{\nu}\|_{L^2(D)}+\|w\|_{L^2(D)}+\|v\|_{L^2(D)})
        \notag\\&\le C(\omega)\|w\|_{L^2(D)}^2+C(\omega)\|w\|_{L^2(D)}\|W_{Q}-z^{\nu}\|_{L^2(D)}+C(\omega)\delta^{\frac{1}{2}}\|w\|_{L^2(D)}
        \notag\\&\le C(\omega)\|w\|_{L^2(D)}^2+\|W_{Q}-z^{\nu}\|^2_{C([0,T];H^4(D))}+\delta
    \end{align}
    and
    \begin{align}\label{a5}
        I_{22}&=2\langle w,(u^{\nu}\cdot\nabla)(W_Q-z^{\nu}+v)\rangle_{L^2(D)}\notag\\&=2\langle w,(u^{\nu}\cdot\nabla)(W_Q-z^{\nu})\rangle_{L^2(D)}+2\langle w,(u^{\nu}\cdot\nabla)v\rangle_{L^2(D)}\notag\\&=:I_{221}+I_{222}.
    \end{align}
    For $I_{221}$ in (\ref{a5}), by applying $(\ref{10})$, we have
    \begin{align}\label{e221}
        I_{221}&\le C\|w\|_{L^2(D)}\|u^{\nu}(t)\|_{C([0,T];L^2(D))}\|W_Q(t)-z^{\nu}(t)\|_{C([0,T];H^4(D))}\notag\\
        &\le C(\omega)\nu\|w\|_{L^2(D)}\|u^{\nu}(t)\|_{C([0,T];L^2(D))}\notag\\&\le C(\omega)\|w\|_{L^2(D)}^2+\nu^2\|u^{\nu}(t)\|^2_{C([0,T];L^2(D))}.
    \end{align}
    From $(\ref{5})$, by applying the Burkholder-Gundy-Davis inequality, we derive that
    \begin{align*}
        \mathbb{E}\sup_{t\in[0,T]}\|u^{\nu}(t)\|_{L^2(D)}^2&\le \|u_0\|^2_{L^2(D)}+T\trace Q_0+2\mathbb{E}\sup_{t\in[0,T]}\left|\int_0^t\langle u^{\nu}(s),dW_Q(s)\rangle_{L^2(D)}\right|\\
        &\le\|u_0\|^2_{L^2(D)}+T\trace Q_0+2\mathbb{E}\left(\int_0^T\|Q^{\frac{1}{2}}_0u^{\nu}(t)\|_{L^2(D)}^2dt\right)^{\frac{1}{2}}
        \\&\le\|u_0\|^2_{L^2(D)}+T\trace Q_0+C\left(\mathbb{E}\int_0^T\|u^{\nu}(t)\|_{L^2(D)}^2dt\right)^{\frac{1}{2}}.
    \end{align*}
    In addition, by taking expectation on both sides of $(\ref{5})$, we obtain
    \[\mathbb{E}\|u^{\nu}(t)\|_{L^2(D)}^2\le\|u_0\|^2_{L^2(D)}+T\trace Q_0,\]
    which implies that
    \begin{align*}
        \mathbb{E}\left(\nu^2\sup_{t\in[0,T]}\|u^{\nu}(t)\|_{L^2(D)}^2\right)\le C\nu^2\left(\|u_0\|^2_{L^2(D)}+T\trace Q_0+1\right)\longrightarrow 0,
    \end{align*}
    as $\nu\to0^+$. This leads to
    \[\lim_{\nu\to 0^+}\nu^2\|u^{\nu}(t)\|^2_{C([0,T];L^2(D))}=0\]
    in probability. For $I_{222}$ given in $(\ref{a5})$, since the integration is taken over a small neighbourhood $\Gamma_{\delta}$ of the boundary $\partial D$ with width being $\delta(\omega)$, where $\delta(\omega)$ is a small random parameter, then we could apply the local coordinate constructed in Section \ref{aaa} to establish smallness for the integration $I_{222}$. Notice that $(\ref{0a})$ ensures
    \[J=1,\]
    where $J$ is the Jacobian of the transformation from Euclidean coordinate to the local coordinate. Hence, we derive with the help of $(\ref{0aa})$ that
    \begin{align}
        I_{222}&= \int_{0}^1\int_0^{\delta} (w_{\tau}h^{-1}u^{\nu}_{\tau}\partial_{\tau} v_{\tau}+w_{\tau}u_n^{\nu}\partial_n v_{\tau}+w_{n}h^{-1}u^{\nu}_{\tau}\partial_{\tau}v_n+w_{n}u^{\nu}_n\partial_{n}v_n)d\alpha ds\notag\\&=:J_1+J_2+J_3+J_4\label{c},
    \end{align}
    where $h$ is given in $(\ref{0aaa})$. By utilizing $(\ref{bb})$, $(\ref{a})$ and $(\ref{b})$, we estimate that
    \begin{align*}
    J_{1}&\le \int_0^1\int_0^{\delta}|w_{\tau}u_{\tau}^{\nu}\partial_{\tau}v_{\tau}|d\alpha ds\\&=\int_0^1\int_0^{\delta}|w_{\tau}(w_{\tau}+u_{\tau}^E-v_{\tau}+z_{\tau}^{\nu}-(W_{Q})_{\tau})\partial_{\tau}v_{\tau}|d\alpha ds\\&\le C(\omega)\|w\|_{L^2(D)}^2+C(\omega)\sqrt{\delta}\|w\|_{L^2(D)}+C(\omega)\|w\|_{L^2(D)}\|z^{\nu}-W_Q\|_{L^2(D)}\\&\le C(\omega)\|w\|^2_{L^2(D)}+C(\omega)\delta+C(\omega)\|z^{\nu}-W_Q\|_{L^2(D)}^2.
    \end{align*}
    The estimates for $J_3$ and $J_4$ given in $(\ref{c})$ are similar, we thus derive 
    \begin{align}\label{f4}J_1+J_3+J_4\le C(\omega)\|w\|^2_{L^2(D)}+C(\omega)\delta+C(\omega)\|z^{\nu}-W_Q\|_{L^2(D)}^2.\end{align}
    As for $J_2$ defined in $(\ref{c})$, the estimate is delicate, since $J_2$ involves with the normal derivative of tangential velocity $\partial_{n}v_{\tau}$ which is large, due to the appearance of boundary layers. By applying $(\ref{b})$, we derive
    \begin{align}\label{d}
        J_2&=\int_0^1\int_0^{\delta}w_{\tau}u_{n}^{\nu}\partial_{n}v_{\tau}d\alpha ds\notag
        \\&=\int_0^1\int_0^{\delta}w_{\tau}(w_{n}+u_n^E-v_n+z_n^{\nu}-(W_Q)_n)\partial_{n}v_{\tau}d\alpha ds\notag\\
        &=\int_0^1\int_0^{\delta}\left(w_{\tau}w_{n}\partial_{n}v_{\tau}+w_{\tau}(u_n^E-v_n)\partial_{n}v_{\tau}+w_{\tau}(z_n^{\nu}-(W_Q)_n)\partial_{n}v_{\tau}\right)d\alpha ds\notag\\&=:J_{21}+J_{22}+J_{23}
    \end{align}
    As for $J_{21}$ in $(\ref{d})$, let us first notice that for a divergence free vector field $v$ which is supported in $\Gamma_{\delta}$, we have 
    \begin{align}\label{9}
        \partial_{\tau}v_{\tau}(s,\alpha)+\partial_{n}(hv_n)(s,\alpha)=0,
    \end{align}
    then by integrating by parts, we obtain
    \begin{align*}
        &J_{21}=\int_0^1\int_0^{\delta}h^{-1}w_{\tau}hw_{n}\partial_{n}v_{\tau}d\alpha ds\\&=-\int_0^1\int_0^{\delta}\partial_{n}(h^{-1}w_{\tau})hw_{n}v_{\tau}d\alpha ds+\int_0^1\int_0^{\delta}h^{-1}w_{\tau}\partial_{\tau}w_{\tau}v_{\tau}d\alpha ds\\&=-\int_0^1\int_0^{\delta}\partial_{n}h^{-1}w_{\tau}hw_{n}v_{\tau}d\alpha ds-\int_0^1\int_0^{\delta}\partial_{n}w_{\tau}w_{n}v_{\tau}d\alpha ds\\&\ \ +\frac{1}{2}\int_0^1\int_0^{\delta}h^{-1}\partial_{\tau}w_{\tau}^2v_{\tau}d\alpha ds.
    \end{align*}
    Hence, with the help of Hardy-Littlewood inequality and $(\ref{b})$, we get
    \begin{align*}
        &J_{21}\le C(\omega)\|w\|_{L^2(D)}^2-\int_0^1\int_0^{\delta}\partial_{n}w_{\tau}w_{n}v_{\tau}d\alpha ds-\frac{1}{2}\int_0^1\int_0^{\delta}\partial_{\tau}(h^{-1}v_{\tau})w_{\tau}^2d\alpha ds
        \\&\le C(\omega)\|w\|_{L^2(D)}^2-\int_0^1\int_0^{\delta}(\partial_{n}u_{\tau}^{\nu}-\partial_{n}u^E_{\tau}+\partial_nv_{\tau}+(\partial_n(W_Q)_{\tau}-\partial_nz_{\tau}^{\nu}))w_{n}v_{\tau}d\alpha ds\\&\le C(\omega)\|\omega\|^2_{L^2(D)}+C(\omega)\delta\|\partial_n u^{\nu}_{\tau}\|_{L^2(\Gamma_{\delta})}\|\nabla w\|_{L^2(D)}+C(\omega)\sqrt{\delta}\|w\|_{L^2(D)}\\&\ \ -\frac{1}{2}\int_0^1\int_0^{\delta} \partial_n v^2_{\tau}w_nd\alpha ds+C(\omega)\|\nabla W_Q-\nabla z^{\nu}\|_{L^2(D)}\|w\|_{L^2(D)},
    \end{align*}
    where
    \begin{align*}
        -\frac{1}{2}\int_0^1\int_0^{\delta} \partial_n v^2_{\tau}w_nd\alpha ds&=\frac{1}{2}\int_0^1\int_0^{\delta}v_{\tau}^2\partial_n(hw_n)h^{-1}+v_{\tau}^2hw_n\partial_n h^{-1}d\alpha ds\\&\le-\frac{1}{2}\int_0^1\int_0^{\delta}v_{\tau}^2\partial_{\tau}w_{\tau}h^{-1}d\alpha ds+C(\omega)\sqrt{\delta}\|w\|_{L^2(D)}
        \\&\le\frac{1}{2}\int_0^1\int_0^{\delta}w_{\tau}\partial_{\tau}(v_{\tau}^2h^{-1})d\alpha ds+C(\omega)\sqrt{\delta}\|w\|_{L^2(D)}\\&\le C(\omega)\sqrt{\delta}\|w\|_{L^2(D)}.
    \end{align*}
    Hence, 
    \begin{align}
        J_{21}&\le C(\omega)\|w\|_{L^2(D)}^2+\delta+C(\omega)\|\nabla W_Q-\nabla z^{\nu}\|_{L^2(D)}^2+C(\omega)\frac{\delta^2}{\nu}\|\partial_n u^{\nu}_{\tau}\|^2_{L^2(\Gamma_{\delta})}\notag\\&+\frac{\nu}{2}\|\nabla w\|_{L^2(D)}^2.\label{f1}
    \end{align}
    For $J_{22}$ given in $(\ref{d})$, notice that $u_n^E-v_n$ vanishes on the boundary, then by applying the mean value theorem, we have
    \[\|u_n^E-v_n\|_{L^{\infty}(\Gamma_{\delta})}\le\delta\|\nabla u_n^E-\nabla v_n\|_{C(\overline{D})}, \]
    which implies
    \begin{align}\label{f2}
        J_{22}\le C(\omega)\delta\|w\|_{L^2(D)}\|\partial_n v_{\tau}\|_{L^2(D)}\le C(\omega)\sqrt{\delta}\|w\|_{L^2(D)}.
    \end{align}
    And for $J_{23}$ defined in $(\ref{d})$, by applying Hardy-Littlewood inequality and the pathwise estimate $(\ref{10})$, we obtain
    \begin{align}\label{f3}
        J_{23}&\le C(\omega)\delta\|\nabla w\|_{L^2(D)}\|\nabla z^{\nu}-\nabla W_Q\|_{L^2(D)}\le C(\omega)\delta\nu\|\nabla w\|_{L^2(D)}\notag\\&\le C(\omega)\delta^2\nu+\frac{\nu}{2}\|\nabla w\|_{L^2(D)}^2.
    \end{align}
    Finally, by combining $(\ref{f4})$, $(\ref{f1})$, $(\ref{f2})$ and $(\ref{f3})$, we arrive at
    \begin{align}\label{e222}
        I_{222}&\le C(\omega)\|w\|^2_{L^2(D)}+C(\omega)\delta+C(\omega)\delta^2\nu+C(\omega)\|z^{\nu}-W_Q\|_{C([0,T];H^4(D))}\notag\\&+C(\omega)\frac{\delta^2}{\nu}\|\partial_nu_{\tau}^{\nu}\|^2_{L^2(\Gamma_{\delta})}+\nu\|\nabla w\|^2_{L^2(D)}.
    \end{align}
    For $I_3$ given in $(\ref{a3})$, we have
    \begin{align}\label{e3}I_3\le C(\omega)\delta^{\frac{1}{2}}\|w\|_{L^2(D)}\le C(\omega)\|w\|_{L^2(D)}^2+\delta.\end{align}
    Combining $(\ref{e1})$, $(\ref{e21})$, $(\ref{e221})$, $(\ref{e222})$ and $(\ref{e3})$, we arrive at
    \begin{align}\label{8}
        \frac{d}{dt}\|w(t)\|_{L^2(D)}^2\le C(\omega)\|w(t)\|_{L^2(D)}^2+R^{\nu}_1(\omega)+R_2^{\nu}(\omega)+R_3^{\nu}(t,\omega),
    \end{align}
    where
    \[R^{\nu}_1(\omega)=C(\omega)\delta+C(\omega)\delta^2\nu+C(\omega)\nu+C(\omega)\nu\delta^{-1}+C(\omega)\|W_Q-z^{\nu}\|_{C([0,T];H^4(D))}^2,\]
    \[R^{\nu}_2(\omega)=\nu^{2}\|u^{\nu}(t)\|_{C([0,T];L^2(D))}^2,\]
    \[R^{\nu}_3(t,\omega)=C(\omega)\frac{\delta^2}{\nu}\|\partial_{n} u^{\nu}_{\tau}\|_{L^2(\Gamma_{\delta})}^2.\]
    To conclude the proof, let us figure out the random parameter $\delta(\omega)$. Define 
    \begin{align}\label{16}\alpha(\omega):=\left(2\nu \int_0^T\|\partial_{n} u^{\nu}_{\tau}\|^2_{L^2(\Gamma_{\delta_0})}ds\right)^{\frac{1}{3}}.\end{align}
    Notice that
    \[\lim_{\nu\to 0^{+}}\alpha=0\]
    in probability, then we set
    \begin{align}\label{17}\delta(\omega):=\min\left\{\frac{\nu}{\alpha},\delta_0\right\}\end{align}
    which satisfies 
    \[\lim_{\nu\to 0^{+}}\frac{\nu}{\delta}=\lim_{\nu\to0^+}\max\left\{\alpha,\frac{\nu}{\delta_0}\right\}=0\]
    in probability. Hence, we have 
    \[\lim_{\nu\to 0^{+}}R_1^{\nu}(\omega)=0\]
    in probability by using $(\ref{10})$. As for $R^{\nu}_3(t,\omega)$, we consider the following two cases.
    
    \noindent\textbf{Case I.} If
    \[\frac{\nu}{\alpha}\le \delta_0,\]
    then from $(\ref{17})$ we have $\delta=\frac{\nu}{\alpha}\le \delta_0$ and thus derive 
    \begin{align*}
        \int_0^TR_{3}^{\nu}(t,\omega)dt&=C(\omega)\frac{\delta^2}{\nu}\int_0^T\|\partial_{n} u^{\nu}_{\tau}\|_{L^2(\Gamma_{\delta})}^2dt\\&\le C(\omega)\frac{\nu}{\alpha^2}\int_0^T\|\partial_{n} u^{\nu}_{\tau}\|_{L^2(\Gamma_{\delta_0})}^2dt=\frac{\alpha C(\omega)}{2}.
    \end{align*}
    
    \noindent\textbf{Case II.} If
    \[\frac{\nu}{\alpha}> \delta_0,\]
    then from $(\ref{17})$ we have $\delta=\delta_0$ and thus derive 
    \begin{align*}
        \frac{\nu^3}{\delta_0^3}>\alpha^3=2\nu \int_0^T\|\partial_{n} u^{\nu}_{\tau}\|^2_{L^2(\Gamma_{\delta_0})}ds,
    \end{align*}
    which implies
    \begin{align*}
        \int_0^TR_{3}^{\nu}(t,\omega)dt\le \frac{\nu C(\omega)}{2\delta_0}.
    \end{align*}
    Summarizing the above results, we conclude that
    \[\int_0^TR_{3}^{\nu}(t,\omega)dt\le\frac{C(\omega)}{2}\max\left\{\alpha,\frac{\nu}{\delta_0}\right\}\longrightarrow0,\]
    in probability from the definition $(\ref{16})$ and the assumption \textit{3(a)} in the Main theorem. By applying Gronwall's inequality on $(\ref{8})$, we conclude 
    \[\|w(t)\|^2_{C([0,T];L^2(D))}\le e^{C(\omega)T}\left(T(R^{\nu}_1(\omega)+R^{\nu}_2(\omega))+\int_0^TR_3^{\nu}(t,\omega)dt\right).\]
    Since the right hand-side converges to 0 in probability as $\nu\to0^+$, then we have 
    \[\lim_{\nu\to 0^{+}}\|\omega(t)\|^2_{C([0,T];L^2(D))}=0\]
    in probability. Moreover, notice that
    \[\|v(t)\|_{C([0,T];L^2(D))}^2\le C(\omega)\delta\le C(\omega)\delta_0\longrightarrow 0,\]
    in probability as $\nu\to 0^+$, then 
    \[\lim_{\nu\to 0^{+}}\|v^{\nu}(t)-v^E(t)\|^2_{C([0,T];L^2(D))}=0\]
    in probability, which implies that 
    \[\lim_{\nu\to 0^{+}}\|u^{\nu}(t)-u^E(t)\|^2_{C([0,T];L^2(D))}=0\]
    in probability from $(\ref{10})$, $(\ref{15})$ and $(\ref{b})$ as desired.
    
    \subsection{Proof of the Main Theorem \textit{(1)} from \textit{3(b)}}
    First, let us notice that the assertions $\textit{3(b)}$ and $\textit{3(c)}$ given in the Main Theorem imply each other, due to the divergence free condition $(\ref{9})$. Hence, it suffices to prove the implication from $\textit{3(b)}$ to $\textit{(1)}$ given in the Main Theorem. The proof presented above from $\textit{3(a)}$ to $\textit{(1)}$ is still applicable, if we modify the method for bounding the term $J_2$ by using the Hardy-Littlewood inequality: 
    \begin{align*}
        J_2&=\int_0^1\int_0^{\delta}w_{\tau}u_{n}^{\nu}\partial_{n}v_{\tau}d\alpha ds\le C(\omega)\delta^{-1} \int_0^1\int_0^{\delta}|w_{\tau}u_{n}^{\nu}|d\alpha ds\\&\le C(\omega)\delta\|\nabla w\|_{L^2(D)}\|\partial_n u^{\nu}_n\|_{L^2(\Gamma_{\delta})}\\&\le C(\omega)\frac{\delta^2}{\nu}\|\partial_n u_n^{\nu}\|^2_{L^2(\Gamma_{\delta})}+\frac{\nu}{2}\|\nabla w\|_{L^2(D)}^2.
    \end{align*}
    This justifies the implication from $\textit{3}(b)$ to $ \textit{(1)}$.

    \vspace{.2in}
    \hspace{-.28in}
    {\bf Acknowledgments:}
This research was partially supported by National Key R\&D Program of China
under Grant No. 2020YFA0712000, National Natural Science Foundation of China under Grant Nos. 12171317, 12161141004 and 12250710674, Strategic Priority Research Program of Chinese Academy of Sciences under Grant No. XDA25010402, and Shanghai Municipal Education Commission under Grant No. 2021-01-
1052 07-00-02-E00087.

    \newpage
    \bibliographystyle{IEEEtran}
    \bibliography{reference}
    
    \end{document}